\documentclass[
11pt,
]{article}

\date{}

\usepackage{geometry}
\usepackage[english]{babel} 
\usepackage[utf8]{inputenc}
\usepackage{amsmath}
\usepackage{amssymb}
\usepackage[hidelinks]{hyperref}
\usepackage{amsthm}
\usepackage{xcolor}

\usepackage[T1]{fontenc}
\usepackage{lmodern}

\usepackage{makeidx}
\usepackage{graphicx}
\usepackage{caption}
\usepackage{subcaption}

\newtheorem{theorem}{Theorem}
\newtheorem*{theorem*}{Theorem}
\newtheorem{defi}{Definition}

\newtheorem{lemma}{Lemma}

\newtheorem{cor}{Corollary}

\newtheorem{claim}{Claim}

\theoremstyle{remark} 

\newtheorem{case}{Case}

\newcommand\numberthis{\addtocounter{equation}{1}\tag{\theequation}}

\newcommand{\q}[1]{q\left( #1 \right)}

\newcommand{\oversetbrace}[2]{\overset{#1}{\overbrace{#2}}}
\allowdisplaybreaks

\begin{document}

\makeatletter\@addtoreset{case}{proof}\makeatother
\title{Improved bounds on the multicolor Ramsey numbers of paths and even cycles}

\author{ 
Charlotte Knierim
	\thanks{
		Department of Computer Science, ETH Zurich, Switzerland.
		Email: {\tt cknierim@ethz.ch}.
	}
\and
Pascal Su
	\thanks{
		Department of Computer Science, ETH Zurich, Switzerland. 
		Email: {\tt sup@inf.ethz.ch}.
	}
}

\maketitle
\begin{abstract}
We study the multicolor Ramsey numbers for paths and even cycles, $R_k(P_n)$ and $R_k(C_n)$, which are the smallest integers $N$ such that every coloring of the complete graph $K_N$ has a monochromatic copy of $P_n$ or $C_n$ respectively. For a long time, $R_k(P_n)$ has only been known to lie between $(k-1+o(1))n$ and $(k + o(1))n$. A recent breakthrough by S\'ark\"ozy and later improvement by Davies, Jenssen and Roberts give an upper bound of $(k - \frac{1}{4} + o(1))n$. 
We improve the upper bound to $(k - \frac{1}{2}+ o(1))n$. 
Our approach uses structural insights in connected graphs without a large matching. These insights may be of independent interest.

\end{abstract}

\section{Introduction}

A classical theorem by Ramsey from 1930~\cite{ramsey1930problem} proves the existence of (finite) Ramsey numbers. The multicolor Ramsey number $R(H_1,H_2,\ldots,H_k)$ is defined as the smallest positive integer $N$ such that for every coloring of the edges of the complete graph $K_N$ with $k$ colors there is a  monochromatic subgraph $H_i$ in the $i$-th color. For convenience, we write $R_k(H)$ if $H_i = H$, for all $i\in [k]$.

 The study of Ramsey numbers started with complete graphs but has been extended to general graphs and has found the interest of many researchers. Several surveys on the topic can be found in~\cite{ conlon2015recent,graham1990ramsey,radziszowski1994small}. Ramsey numbers for complete graphs are known to grow exponentially and are hard to analyze. Already the base of the exponential term is very much unknown, the most recent bound was by Conlon~\cite{conlon2009new}.

One of the most natural families of graphs to study are paths $P_n$ and cycles $C_n$. The two-color Ramsey numbers for paths are known since 1967~\cite{gerencser1967ramsey}. For more colors the problem is more difficult and it took until 2007 for Gy\'arf\'as, Ruszink\'o, S\'ark\"ozy and Szemer\'edi to prove that $R_3(P_n) = 2n + \mathcal{O}(1)$~\cite{Gyarfas2007}. This progress was initiated with the use of Szemer\'edi's regularity lemma~\cite{szemeredi1975regular}. The regularity lemma has proven to be a strong tool in many applications of extremal graph theory. Figaj and Łuczak~\cite{figaj2007ramsey} used it to show that in the case of Ramsey numbers asymptotically, avoiding \textit{connected matchings}, paths and cycles is the same (see Lemma~\ref{lem:CycMat} below). A matching $M$ is connected in $G$ if all edges
of $M$ are in the same component of $G$. The idea of using these connected matchings was suggested by Łuczak~\cite{luczak1999r}. This method has since been used in a series of papers (see e.g.~\cite{benevides20093,figaj2007ramsey,gyarfas2007three,luczak2012multi}).

 These results show that finding large connected matchings is an essential step towards understanding the Ramsey numbers of paths and even cycles.  
 
  For $k$ colors, an easy application of the Erd\H{o}s-Gallai extremal theorem on each color class gives a bound of $R_k(P_n) \le k n$ for even $n$ . 
 This was best known until a recent breakthrough by S\'ark\"ozy~\cite{sarkozy2016multi}, who further improved this to $(k- \frac{k}{16 k^3 +1} )n$. Davies, Jenssen and Roberts~\cite{davies2017multicolour} refined his ideas to get $R_k(P_n) \le \left( k-\frac{1}{4} + \frac{1}{2k} \right) n$. 
 
 All of these recent results focus on the $k$-color Ramsey number for paths and obtain the bounds for connected matchings in similar ways. As graphs without connected matching seem to be conceptually easier to analyze than graphs without paths our main improvement comes from considering the former.

\paragraph{Contribution} In this paper, we analyze the structural properties of graphs not containing large connected matchings as a subgraph. We show the vertices can be very clearly partitioned into categories of `low', `intermediate' and `high' degree vertices. In particular, if a connected component is large compared to $n$, then only a small number of vertices have high degree. We introduce a tool for counting edges in such graphs. 

While the arguments in the two most recent papers on $R_k(P_n)$ revolve around the overlap of dense connected components, we add the analysis on large connected components. This is important because the construction for the lower bound by Yongqi, Yuansheng, Feng and Bingxi~\cite{yongqi2006new}, who showed that for all $k\ge 3$ we have 
$R_k(C_n) \ge  \left( k-1 +o(1) \right) n ,$  also contains these large components. Note that the bound is given for cycles but easily extends to paths. We believe this is an essential step towards achieving exact bounds for the multicolor Ramsey numbers for paths and cycles.

For this, we consider the following variant of finding monochromatic connected matchings in $k$-colored dense graphs. This will allow us to generalize our results to paths and cycles using the regularity lemma. 

\begin{theorem}
	Let $k\geq 4$ be a positive integer and  $\varepsilon,\delta$ constants, such that $0< \varepsilon \leq 1/2$ and  $0\leq\delta<\frac{\varepsilon^3}{3k^2}$. Then for every even integer $n\geq 4$, every $k$-colored graph $G$ with $v(G)> (k-1/2+\varepsilon)n$ and $e(G)\geq (1-\delta)\binom{v(G)}{2}$ has a monochromatic connected matching of size $n/2$. 
	\label{thm:main}
\end{theorem}

We will then use the mentioned relationship between avoiding even cycles and connected matching to derive the following theorem.

\begin{theorem}
	For every integer $k\geq 4$ and an even integer $n$
	\[R_k(C_n)\leq \left(k-\frac{1}{2}+o(1)\right)n.\]
	\label{thm:RCn}
\end{theorem}

As as immediate consequence of the above result we obtain an upper bound on the $k$-color Ramsey numbers for paths which follows from the fact that $P_n\subseteq C_n$.

\begin{cor}
	For every integer $k\geq 4$ and an even integer $n$
	\[R_k(P_n)\leq \left(k-\frac{1}{2}+o(1)\right)n.\]
	\label{cor:RPn}
\end{cor}

\paragraph{Outline.} Our paper is organized as follows. In Section~\ref{sec:methods} we show the structural results for graphs avoiding connected matchings and two tools to apply these structures for showing Ramsey results.
In Section~\ref{sec:mainproof} we prove the main result in two steps. First, we use our structural results from
the preceding section to show that we can only have few high-degree vertices. Second, we remove
these vertices and show that the resulting connected components are then all small (of size at most n).
In Sections~\ref{sec:structure} and~\ref{sec:LemlossCorrloss} we prove the lemmas we used in Section~\ref{sec:methods}. 
Finally, in Section~\ref{sec:conclusion} we summarize our results and give some ideas for future work.

\section{Methods}\label{sec:methods}

All graphs considered in this paper are simple, without loops or multiple edges. For a graph $G=(V,E)$ we denote by $V(G)$ and $E(G)$ the vertex set and the edge set of the graph $G$ and we set $v(G)$ and $e(G)$ to be the respective cardinalities. 

As mentioned there is a relation between avoiding connected matchings and even cycles in form of the following lemma which is a variant of a lemma used in~\cite{figaj2007ramsey}.  
\begin{lemma}[Lemma 8 in~\cite{luczak2012multi}]
	Let a real number $c>0$ be given. If for every $\varepsilon >0$ there exists a $\delta >0$ and an $n_0$ such that for every even $n>n_0$ and any graph $G$ with $v(G)>(1+\varepsilon)cn$ and $e(G)\geq (1-\delta)\binom{v(G)}{2}$ and any $k$-edge-coloring of $G$ has a monochromatic component containing a matching of $n/2$ edges then
	\[R_k(C_n)\leq \left(c+o(1)\right)n.\]
	\label{lem:CycMat}
\end{lemma}
For $c\geq 1$ we surely have $(1+\varepsilon)cn \geq (c+\varepsilon)n$ which is more convenient for later calculations. Observe that Theorem~\ref{thm:RCn} follows from Theorem~\ref{thm:main} by using Lemma~\ref{lem:CycMat} with $c=k-1/2$.\\

 \subsection{Structure}

The structure of graphs without large connected matchings play an important role in this paper. As a path on $n$ vertices contains a connected matching on $\lfloor n/2 \rfloor$ edges, extremal results for paths directly give an upper bound for connected matchings.

\begin{lemma}[Erd\H{o}s-Gallai~\cite{erdHos1959maximal}]
	Let $H$ be a graph which does not contain an $n$-vertex path. Then \[e(H)\leq \frac{n-2}{2}v(H).\]
\label{lem:ErdösGallai}
\end{lemma}
\begin{cor}
	Let $H$ be a graph which does not contain a connected matching of size $n/2$ for even $n$. Then 
	\[e(H)\leq \frac{n-2}{2}v(H).\]
\label{cor:ErdösGallai}
\end{cor}
The extremal graph in both cases consists of disjoint cliques of size $n-1$.
In~\cite{balister2008connected} the extremal configurations of connected graphs without a long path are discussed. We provide a structure for connected graphs without large connected matchings which is very similar. We capture this in the following lemma.
\begin{lemma}
	For every connected graph $G=(V,E)$ without a matching of size $n/2$ there is a partition $S\cup Q\cup I$ of the vertex set such that \begin{enumerate}
		\item $|Q|+ 2|S| =  min\{v(G),n-1\} $, \label{lemstruc:Cond:1}
		\item $I$ is an independent set; additionally, if $v(G)\leq n-1$, then $I=\emptyset$,\label{lemstruc:Cond:2}
		\item every vertex in $Q$ has at most one neighbor in $I$,\label{lemstruc:Cond:3}
		\item every vertex in $I$ has degree less than $n/2$. \label{lemstruc:Cond:4}
	\end{enumerate}
	\label{lem:structure}
\end{lemma} 
 We only want to give an intuition on the structure here and defer its proof to Section~\ref{sec:structure}. In Figure~\ref{fig:structure} an example of such a partition can be seen.
 Our structure has a series of properties that we use later. From $|Q|+ 2|S| \leq n - 1$  we can easily deduce that $|S| < n/2$ and $|Q| < n$. From $|Q|+ 2|S| \leq v(G)$ we get $ |S|\le |I|$, with a strict inequality as soon as the number of vertices exceeds $n-1$. Note that the graph induced by $Q$ and $S$ can potentially be a clique, and we think of the vertices in $S$ as high degree vertices because they can potentially have edges to every other vertex. Vertices in $Q$ have at most $|S|+|Q| < n$ neighbors and if $v(G) < n$, then $I$ is empty so $|S| = 0$ and $|Q| = v(G)$ (using \eqref{lemstruc:Cond:1}).
\begin{figure}
\centering
\includegraphics[scale=0.8]{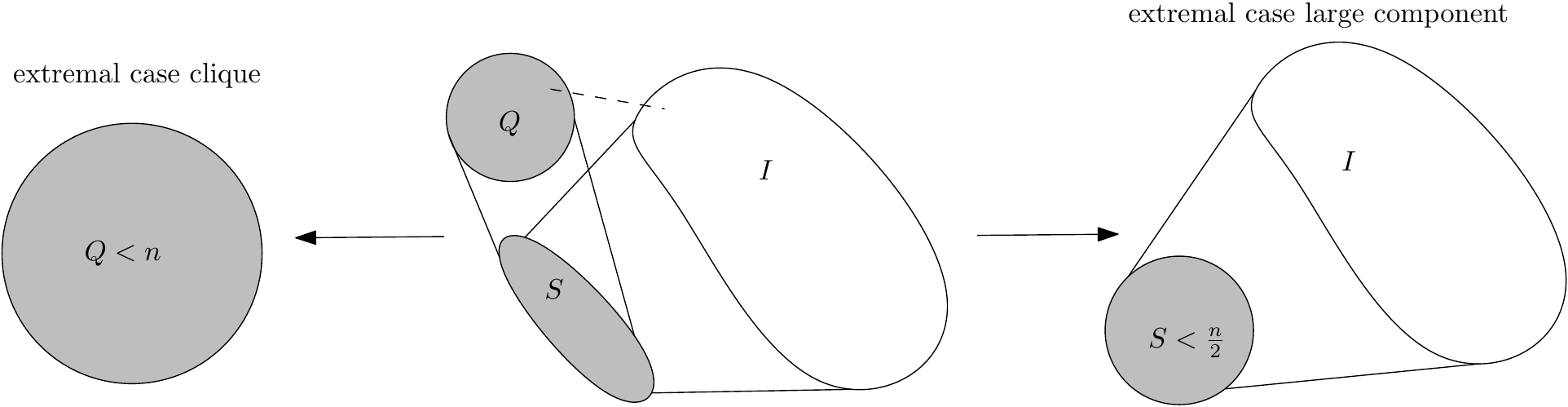}
\caption{Visualization of the structure from Lemma~\ref{lem:structure} including extremal structures}
\label{fig:structure}
\end{figure}
  
From the above properties it also follows that the number of edges is at most $\binom{|Q|+|S|}{2} + |I| \cdot |S| + |Q|$. After some careful consideration one can see that this bound is stronger than the Erd\H{o}s-Gallai bound as soon as $v(G) > n+1$.

 \subsection{Loss Function}
 
 In the following we define a function which intuitively should capture the difference between the bound of Erd\H{o}s-Gallai and the bound implied by Lemma~\ref{lem:structure}. Remember Corollary~\ref{cor:ErdösGallai} gave us $e(G)\leq \frac{n-2}{2}v(G)$ for general graphs. To simplify computations, we use the bound
  \[e(G)\leq \frac{n-1}{2}v(G).
 \]
 We denote by $\mathcal{G}$ the class of all graphs without a connected matching of size $n/2$. Then the function maps $\mathcal{G}$ to a positive rational number. 
 
 It is important to note that whenever we have connected components we can partition them as in Lemma~\ref{lem:structure}. This partition may not be unique but we can fix an arbitrary one to make the definitions consistent.

 \begin{defi}
 	Let $\mathcal{G}$ be the class of all graphs without a connected matching of size $n/2$. We define $f\colon\mathcal{G}\to \mathbb{Q}_{\geq 0}$ as the  difference between $\frac{n-1}{2}v(G)$ and the number of edges in $G$, i.e.
$$ f(G)=\frac{n-1}{2}v(G) -e(G).$$  	\label{def:fG}
 \end{defi}
 Observe that Corollary~\ref{cor:ErdösGallai} guarantees that $f(G)\geq 0$ for any graph without a connected matching of size $n/2$. One way to think about $f(G)$ is as the number of edges $G$ loses in comparison to $\frac{n-1}{2}v(G)$. In the following we also refer to this value as the loss of $G$. Intuitively, this loss can happen for two reasons:
 \begin{enumerate}
 	\item $G$ can lose edges because of large connected components. This part can be captured by the difference between $\frac{n-1}{2}v(G)$ and the bound of edges for connected graphs obtained from Lemma~\ref{lem:structure}. 
 	\item Edges missing because $G$ is not saturated. These are the edges that are not present in $G$ but could be added without creating a connected matching of size $n/2$. 
 \end{enumerate}
 
 Throughout this paper we only look at the number of lost edges, it will not be important which exact edges are lost or what the cause was. 
 We introduce a function distributing this loss among the vertices of $G$. 
 
 \begin{defi}
 	Let G be a graph without a connected matching of size $n/2$. We denote by $f(v)\colon V(G)\to \mathbb{Q}_{\geq0}$ the loss of edges in $G$ compared to $\frac{n-1}{2}$ caused by a single vertex $v\in V(G)$ defined as follows. Let $C$ be the connected component of $G$ including $v$ and let $S\cup Q \cup I$  be a partition of $V(C)$ as described in Lemma~\ref{lem:structure}. We set 
 	\begin{align*}
	 	f(v)=\begin{cases}
		 	\frac{n-1}{4},&\text{if } v\in S,\\
		 	\frac{n-1}{2}-\frac{\deg(v)}{2},&\text{if }v\in Q,\\
		 	0,& \text{if }v\in I.
	 	\end{cases}
 	\end{align*}\\
 	\label{def:fvW}
 \end{defi}

From Lemma~\ref{lem:structure} we have that $\deg(v) \le n-1$ for vertices in $Q$ and so the function $f(v)$ is non-negative and well defined. Since we want to distribute the loss of the graph over the different vertices we need that the sum of the losses over all vertices is at most the loss of $G$. We claim that this holds for $f$ defined as above. 
 \begin{lemma}
 	Let $G$ be a graph without a connected matching of size $n/2$. 
 	Then $$\sum_{v\in V(G)}f(v)\leq f(G).$$\label{lem:fconnloss}
 \end{lemma}
 We defer the proof of this to Section~\ref{sec:LemlossCorrloss}.\\
 
 Let $G$ be a graph with a $k$ coloring of its edges. We introduce three different categories for the vertices of $G$. Let $v\in V(G)$ be some vertex of our graph and let $G_1,\ldots,G_k$ be the monochromatic subgraphs in each color. Then for every color $i$ the vertex $v$ is in exactly one component of $G_i$. We denote this component by $C_i^v$. Then we have $C_1^v,\ldots,C_k^v$ which are $k$ monochromatic components in different colors, each including the vertex $v$. We partition the vertices into three classes depending on their role in the $k$ components.
\begin{defi}
	For a graph $G$ with a $k$-edge-coloring not containing a monochromatic connected matching of size $n/2$ and a vertex $v\in V(G)$, let $C_1^v,\ldots,C_k^v$ be the $k$ components containing $v$ in $G_1,\ldots ,G_k$. Consider the partition $C_i^v=S_i^v\cup Q_i^v \cup I_i^v$ given by Lemma~\ref{lem:structure} for every color $i$.
	We call the vertex $v$
\begin{enumerate}
	
	\item \textit{\textbf{strong}}\text{ if it is in $S_i^v$ for some color $i$,}\label{def:strong}
	\item \textit{\textbf{Q-saturated}}\text{ if it is in $Q_i^v$ for every color $i$,}\label{def:qsat}
	\item \textit{\textbf{small}} \text{ if it is in $I_i^v$ for some color $i$ and in $Q_j^v$ or $I_j^v$ for all other colors $j$.\label{def:small}}
\end{enumerate}
\label{defi:Vertexclass}
\end{defi}
Generalizing Definition~\ref{def:fG}, we define $F(G)$ to be the total loss of edges over all colors in a graph $G$ without a monochromatic connected matching of size $n/2$. We consider $G$ to be the union of $k$ monochromatic graphs, each of which avoids a connected matching of size $n/2$. Applying Corollary~\ref{cor:ErdösGallai}, again with a slight weakening of the bound, for each of the color classes, gives 
	\[e(G) \leq k\cdot \frac{n-1}{2}v(G).\]
This is the bound we use for comparison throughout the paper.
	\begin{defi}
		Let $\mathcal{G}_c$ be the class of all $k$-edge-colored graphs which do not contain a monochromatic connected matching of size $n/2$ and let $F\colon\mathcal{G}_c\to \mathbb{Q}_{\geq 0}$ be defined as the difference between $k\cdot\frac{n-1}{2}v(G)$ and the number of edges in $G$, i.e.
		$$	F(G)=k\cdot\frac{n-1}{2}v(G)-e(G).$$
	\end{defi} 
Observe that $F(G) =k\cdot\frac{n-1}{2}v(G)-e(G) = \sum_{i=1}^{k} f(G_i).$
The above equality holds because every edge of $G$ is colored in exactly one color and thus counted in exactly one $G_i$. This also implies that $F(G)$ is non-negative by the non-negativity of $f(G_i)$.
	 
We again distribute this loss over the vertices in $G$.
\begin{defi}
	Let $G$ be a $k$-edge-colored graph without a monochromatic connected matching of size $n/2$ and let $F(v)\colon V(G)\to \mathbb{Q}_{\geq 0}$ denote the loss of edges in $G$ compared to $k\cdot \frac{n-1}{2}v(G)$ caused by a single vertex $v\in V(G)$. We set
	\begin{align*}
	F(v)=\begin{cases}
	\frac{n-1}{4},& \text{if $v$ is strong},\\
	k\cdot \frac{n-1}{2} -\frac{\deg(v)}{2},&\text{if $v$ is $Q$-saturated}, \\ 
	0, & \text{if $v$ is small}.
	\end{cases}
	\end{align*}
	\label{def:lossfctF}
\end{defi}
This is again well defined as a $Q$-saturated vertex has degree at most $n-1$ in every color by Lemma~\ref{lem:structure}. We conclude $\deg(v)\leq k\cdot(n-1)$ which directly implies that $F(v)$ is non-negative.
We deduce the following corollary from Lemma~\ref{lem:fconnloss} to again verify that we have a valid distribution of the loss. The proof is deferred to Section~\ref{sec:LemlossCorrloss}.
\begin{cor}
	Let $G$ be a $k$-edge-colored graph  without a monochromatic matching of size $n/2$. Then
	\[\sum_{v\in V(G)}F(v)\leq F(G).\]
	\label{cor:lossFv}

\end{cor}

\subsection{Small Components}

Finally, we introduce a lemma to find an upper bound on the number of edges of graphs whose color classes have the special property of consisting only of components which are not too large. 
\begin{lemma}
	For any two integers $k\geq 4$ and $n\geq 4$ let $G$ be a $k$-edge-colored graph on $(k-1/2)n$ vertices with color classes $G_1,\ldots,G_k$  such that no $G_i$ has a component of size larger than $n$.
	Then we have
	\[e(G)\leq \binom{v(G)}{2}-\frac{n^2}{32}.\]
	\label{lem:eGsmallcomp}
\end{lemma}
\begin{proof}
	
	For sake of contradiction we assume there exists a graph $G$ on $(k-1/2)n$ vertices with a $k$-coloring such that no color class contains a component of size larger than $n$ and 
	 \[e(G) > \binom{v(G)}{2}-\frac{n^2}{32}.\numberthis\label{eq:eGsmallcontra}\]  
	 Additionally, let $G$ be the graph such that the number of edges is maximal among all graphs satisfying the property above.
	Choose any color from the graph and without loss of generality we assume it is blue.
	 Over all possible colorings satisfying that no color class has a component of size larger than $n$ we choose the coloring which maximizes the number of blue edges, i.e.~blue is the densest color. Let $G_B$ be the subgraph of $G$ induced by the blue edges. We take a closer look at the structure of the blue graph. 
	
	\begin{claim}
		The number of components in $G_B$ is either $k$ or $k+1$.
	\end{claim}
	\begin{proof} 
		 If we had only $k-1$ components, then we could cover at most $(k-1)n<(k-1/2)n$ vertices. Hence we can be sure we have at least $k$ components. By edge maximality of $G$ and the choice of the coloring we cannot have more than one component of size less than $n/2$ in the blue graph, as otherwise we could add or recolor any edge between the two components. But if we had $k+2$ components, then by convexity and the fact that there is at most one component of size less than $n/2$, the number of edges is maximized when we have $k-3$ cliques of size $n$ , one clique of size $n-1$, $3$ cliques of size $n/2$ and one isolated vertex. 
		 Then 
		\begin{align*}
		e(G_B)&\leq(k-2)\binom{n}{2}+3\binom{n/2}{2}\leq (k-2)\frac{n^2}{2} + 3\frac{n^2}{8} = \left(k-\frac{5}{4} \right) \frac{n^2}{2}.
		\end{align*}
		
		As blue was the densest color, the total number of edges is at most $k\cdot e(G_B)$. We conclude

		$$e(G)\leq k\left(k-\frac{5}{4} \right) \frac{n^2}{2}.$$

		A simple algebraic transformation gives
		$$e(G)\leq\left(k^2-k+\frac{1}{4}-\frac{k-1/2}{n}\right)\frac{n^2}{2}-\left(\frac{k}{4}+\frac{1}{4}-\frac{k-1/2}{n}\right)\frac{n^2}{2}.$$
	Using that the complete graph on $(k-1/2)n$ vertices has $e(K_{v(G)})=\binom{v(G)}{2}=\frac{n^2}{2}\left(k^2-k+\frac{1}{4}-\frac{k-1/2}{n}\right)$ edges and assuming $n\geq 4$ we get
		$$e(G)\leq \binom{v(G)}{2}- \frac{n^2}{16}.$$

	 Thus, $G$ having at least $k+2$ components contradicts our assumption \eqref{eq:eGsmallcontra} so we conclude that $G_B$ has either $k$ or $k+1$ components.
		
	\end{proof}
	By convexity and our previous observation that the blue graph has either $k$ or $k+1$ components, we obtain the number of blue edges is maximized when there are $k$ components: $k-1$ of size $n$ and one of size $n/2$. In this case we have 
	\begin{align*}
	e(G_B)&\leq (k-1)\frac{n^2}{2}+\frac{n^2}{8}=\left(k-\frac{3}{4}\right)\frac{n^2}{2}.
	\end{align*}
	
Let us now see what this implies for the other components. Consider the color with the second most edges, say, red. Let $G_R$ be the subgraph of $G$ induced by the red edges. By the choice of $G$ and the coloring, the blue components must be cliques. 
Thus for all components $C_B\subseteq G_B$ and $C_R\subseteq G_R$, the edges with both endpoints in $V(C_B)\cap V(C_R)$ are blue. Observe that as the blue color has at most $k+1$ components, this implies that we can view the red graph as a union of $(k+1)$-partite components, where the number of edges is maximized when all components are complete $(k+1)$-partite where every part has the same size. We observe that by the same reasons as seen for the blue graph, the red graph cannot have fewer than $k$ components.
	We conclude that the number of edges in the red graph is maximized, when there are again $k$ components: $k-1$ of size $n$ and one of size $n/2$. As $G_R$ is $(k+1)$-partite we know that a component of size $n$ has at most $\left(\left(1-\frac{1}{k+1}\right)\frac{n^2}{2}\right)$ edges. We get
	\begin{align*}
	e(G_R)&\leq \left(1-\frac{1}{k+1}\right)\left((k-1)\frac{n^2}{2}+\frac{1}{4}\frac{n^2}{2}\right)\\
	&\leq\left(1-\frac{1}{k+1}\right) \left( k-\frac{3}{4}\right) \frac{n^2}{2}.
	\end{align*}
	
	As red was the second densest color, the number of edges in every other color is at most the number of red edges. We conclude for the total number of edges	
	
	\begin{align*}
	e(G)&\leq e(G_B)+(k-1)\cdot e(G_R)\\
	&\leq \left(k-\frac{3}{4}\right)\frac{n^2}{2}+(k-1)\left(1-\frac{1}{k+1}\right) \left( k-\frac{3}{4}\right) \frac{n^2}{2}.
	\end{align*}	
	Simplifying gives    
	\begin{align*}
	e(G)\leq \left(k^2-\frac{7}{4}k+\frac{11}{4}-\frac{7}{2(k+1)}\right)\frac{n^2}{2}.
	\end{align*}
	Using the fact that the graph $G$ has $(k-\frac{1}{2})n$ vertices, $n\geq 4$ and $k\geq 4$, we conclude the proof of the lemma with
	$$e(G)\leq \binom{v(G)}{2}-\frac{n^2}{32}.$$

\end{proof}
\section{Proof of the Main Theorem}\label{sec:mainproof}

The main part of this paper is the proof of Theorem~\ref{thm:main} about finding monochromatic connected 
matchings in $k$-edge-colored dense graphs. We restate the theorem for convenience.\\
\\
\textbf{Theorem 1.} \textit{
	Let $k\geq 4$ be a positive integer and  $\varepsilon,\delta$ constants, such that $0< \varepsilon \leq 1/2$ and  $0\leq\delta<\frac{\varepsilon^3}{3k^2}$. Then for every even integer $n\geq 4$, every $k$-colored graph $G$ with $v(G)> (k-1/2+\varepsilon)n$ and $e(G)\geq (1-\delta)\binom{v(G)}{2}$ has a monochromatic connected matching of size $n/2$.  }\\

In the previous section we stated all the tools we need to prove Theorem~\ref{thm:main}, so all that remains is to put all the lemmas together.

\begin{proof}[Proof of Theorem~\ref{thm:main}]
For $0< \varepsilon \leq \frac{1}{2}$, let $\alpha=1/2-\varepsilon$ and $\delta<\frac{\varepsilon^3}{3k^2}$. We proceed to prove Theorem~\ref{thm:main} by contradiction. Let $G$ be an edge maximal graph on $(k-\alpha)n$ vertices with the property that there is a $k$-coloring of the edges of $G$ that avoids a monochromatic connected matching of size $n/2$. 
 Assume $G$ has many edges, \[e(G)\geq (1-\delta)\binom{v(G)}{2} > \binom{v(G)}{2}-\delta k^2\frac{n^2}{2}.\label{eq:maincontra}\numberthis\]
 We then derive a contradiction by showing that $G$ cannot have enough edges. 

Firstly, we show that the number of vertices of low degree in $G$ is small. Let $V_{\ell}$ be all vertices $v\in V(G)$ with $\deg(v)<(k-1/2)n$. 
\begin{claim}
	$|V_{\ell}|\leq \frac{\delta k^2 n}{\varepsilon}.$
	\label{cl:main1}
\end{claim}
\begin{proof}

Observe that every vertex in $V_\ell$ misses at least $\varepsilon n$ incident edges. Thus we can find a lower bound on the number of edges missing in $G$ compared to the complete graph $K_{v(G)}$ by $\frac{|V_{\ell}|\varepsilon n}{2}$. As we cannot miss more than $\delta k^2\frac{n^2}{2}$ edges we conclude $|V_{\ell}|\leq \frac{\delta k^2 n}{\varepsilon}.$

\end{proof}

Observe next that every vertex in the set $V(G)\setminus V_{\ell}$ has degree at least $(k-1/2)n$. Recall that, by Lemma~\ref{lem:structure}, we can partition every color class $G_i$ into $S_i\cup Q_i\cup I_i$, such that all vertices with very large degree are in $S_i$ and all vertices in $I_i$ have degree less than $n/2$ and all vertices in $Q_i$ have degree less than $n$. Let $v\in V(G)\setminus V_{\ell}$ be a vertex which is in the $I_i$ for some color $i$ and thus has degree less than $n/2$ in this color. If this vertex was not in $S_j$ for any color $j$, then it has a maximum degree of $n-1$ in every other color, thus $\deg(v)< n/2 +(k-1)(n-1) < (k-1/2)n$,  contradicting that $v\in V(G)\setminus V_{\ell}$. We conclude that no vertex in $V(G)\setminus V_{\ell}$ by Definition~\ref{defi:Vertexclass} is small and thus every vertex in $V(G)\setminus V_{\ell}$ is either strong or $Q$-saturated, meaning it is in $S_i$ for some color $i$ or it is in $Q_i$ for all colors $1\leq 1\leq k$. 

Secondly, we show that the number of strong vertices in $G$ also has to be small. Let $V_s$ be the set of all strong vertices in $V(G)\setminus V_{\ell}$ and let $\beta=\frac{|V_s|}{n}$. 
\begin{claim}
	$\beta\leq\frac{2\delta k^2}{\varepsilon^2}. $
	\label{cl:main2}
\end{claim}
\begin{proof}
Remember that we previously saw that we can use the function $F(v)$ from Definition~\ref{def:lossfctF} to capture the loss of edges in a graph caused by vertex $v\in V(G)$. By definition for $Q$-saturated vertices we have 
\begin{align*}
	F(v)&= k\cdot\frac{n-1}{2} -\frac{\deg(v)-1}{2}.\\
\end{align*}
As $\deg(v)\leq v(G)-1= (k-\alpha)n-1$ we conclude that for a $Q$-saturated vertex $v\in V(G)$
\begin{align*}
F(v)&\geq k\cdot\frac{n-1}{2} - \frac{(k-\alpha)n-1}{2}=\alpha\frac{n}{2}-\frac{k-1}{2}.\numberthis\label{eq:FQsatmain}
\end{align*}
By definition of $F$ we have $F(G)=k\cdot\frac{n-1}{2} v(G)-e(G)$  and Corollary~\ref{cor:lossFv} states $\sum_{v\in V(G)}F(v)\leq F(G)$, so we can bound the number of edges in $G$ by
\[e(G)\leq k\cdot\frac{n-1}{2} v(G)- \sum_{v\in V(G)}F(v).\]

Next, we use the loss function $F$ of the vertices to find an upper bound on the number of edges in $G$. All small vertices in $G$ have to be in $V_{\ell}$ and thus there cannot be too many of them, remember for strong vertices we have $F(v)=\frac{n-1}{4}$ by Definition~\ref{def:lossfctF}. Using this, Equation~\eqref{eq:FQsatmain} for $Q$-saturated vertices and the fact that there are only few small vertices allows us to find an upper bound on the number of strong vertices. We get
\begin{align*}
	e(G) &\leq k\cdot\frac{n-1}{2}\left(k-\alpha\right)n -\sum_{v\in V(G)} F(v)\\
	&\leq k\cdot\frac{n-1}{2}(k-\alpha)n -\oversetbrace{\substack{\text{vertices in } V(G)\setminus V_{\ell}\\\text{which are $Q$-saturated }}}{\left(k-\alpha -\beta - \frac{\delta k^2}{\varepsilon}\right)n\left(\alpha\frac{n}{2}-\frac{k-1}{2}\right)}-\oversetbrace{\substack{\text{vertices in } V(G)\setminus V_{\ell}\\\text{which are strong}}}{\beta n\left(\frac{n-1}{4}\right)}.
	%
\end{align*}
Now it follows from algebraic transformations and simple estimations that
\begin{align*}
	e(G) &\le \frac{n^2}{2}\left(k^2-2\alpha k +\alpha^2 -\frac{k-\alpha}{n}\right) -\frac{n^2}{2} \left(\frac{\beta}{2}-\alpha\beta -\frac{\delta k^2}{\varepsilon}\alpha\right).\\
\end{align*}
Observe that $\binom{v(G)}{2}=\frac{n^2}{2}\left(k^2-2k\alpha +\alpha^2-\frac{k-\alpha}{n}\right)$. Together with $\alpha=1/2-\varepsilon$ this gives
	$$e(G)\leq\binom{v(G)}{2}-\frac{n^2}{2}\left(\varepsilon\beta-\frac{\delta k^2}{\varepsilon}\alpha\right).$$
By the assumption $e(G) \ge \binom{v(G)}{2}-\frac{ \delta k^2  n^2}{2}$ we conclude that $ \varepsilon\beta-\frac{\delta k^2}{\varepsilon}\alpha \leq \delta k^2.$
Using $\alpha \leq 1$ and $\varepsilon \leq 1$ we get a bound on the size of $\beta$.
$$	\beta \leq  \frac{\delta k^2}{\varepsilon} + \frac{\delta k^2}{\varepsilon^2}\alpha\leq \frac{2\delta k^2}{\varepsilon^2} .$$
This concludes the proof of Claim 2. 
\end{proof}
We look at the induced subgraph on the remaining vertices i.e.~$V(G)\setminus\{V_\ell\cup V_s\}$. Observe that we get using that $\delta<\frac{\varepsilon^3}{3k^2}$, $\alpha=1/2-\varepsilon$ and Claims~\ref{cl:main1} and~\ref{cl:main2}
\begin{align*}
	\left|V(G)\setminus\left(V_{\ell}\cup V_s\right)\right|& \geq v(G) - \frac{\delta k^2n}{\varepsilon} -\frac{2\delta k^2n}{\varepsilon^2} \\\
	&\geq v(G) -  \frac{3\delta k^2n}{\varepsilon^2}\\
	&> v(G)- \varepsilon n = (k-1/2)n.
\end{align*}

We now remove $\varepsilon n$ vertices from the graph, including the vertices from $V_s$ and $V_{\ell}$. We are left with a graph $G'$ on $(k-1/2)n$ vertices such that every vertex was $Q$-saturated in $G$. This means we remove from all components in all colors $i$ the vertices in $I_i$ and $S_i$ of the corresponding partition and we are thus left with only monochromatic components of size at most $n-1$. This allows us to apply Lemma~\ref{lem:eGsmallcomp} which gives us $e(G')\leq \binom{v(G')}{2}-\frac{n^2}{32}\leq \binom{v(G')}{2}- \delta k^2 \frac{n^2}{2}$. Which follows by the choice of $\varepsilon$ and $\delta$.

From here we conclude that there are at least $\delta k^2 \frac{n^2}{2}$ edges not present in $G'$. But then, even without considering the edges missing in the rest of $G$ we can conclude that $e(G)\leq \binom{v(G)}{2}-\delta k^2 \frac{n^2}{2}<(1-\delta)\binom{v(G)}{2}$ edges, yielding the desired contradiction.

\end{proof}

\section{Proof of Lemma~\ref{lem:structure}}

\label{sec:structure}

In the following we derive some properties of connected graphs without large matchings. For convenience we restate the lemma\\
\\
\textbf{Lemma~\ref{lem:structure}.} \textit{For every connected graph $G=(V,E)$ without a matching of size $n/2$ there is a partition $S\cup Q\cup I$ of the vertex set such that \begin{enumerate}
		\item $|Q|+ 2|S| =  min\{v(G),n-1\} $, 
		\item $I$ is an independent set; if $v(G)\leq n-1$, then $I=\emptyset$,
		\item every vertex in $Q$ has at most one neighbor in $I$,
		\item every vertex in $I$ has degree less than $n/2$. 
	\end{enumerate}}
	\vspace{1em}
For a graph $G$ and $A \subseteq V(G)$ we define $G\setminus A$ as the subgraph of $G$ where all vertices from $A$ and their incident edges are removed.
 For a matching $M\subseteq E$ we call all vertices which are not incident to any edge in $M$ the unmatched vertices in $M$. Furthermore, we denote by $q(G\setminus S)$ the number of odd components in $G\setminus S$. 
We use a generalization of Tutte's Theorem by Berge~\cite{berge1958matchings} in our proof.
\begin{theorem}[\textbf{Berge~\cite{berge1958matchings}}]
	Let $G=(V,E)$ be a graph. For any set $S\subseteq V$ and any matching $M$, the number of unmatched vertices in $M$ is at least $\q{G\setminus S}-|S|$. Moreover, there exists a set $S\subset V$ such that every maximum matching of $G$ misses exactly $\q{G\setminus S} - |S|$ vertices.
	\label{thm:TutteBerge}
\end{theorem}
\begin{proof}[Proof of Lemma~\ref{lem:structure}]
We distinguish two cases in the proof.
\begin{case}[$v(G)\leq n-1$]
	In this case we set $Q=V$ and $S=I=\emptyset$. It can be easily verified that all four conditions are satisfied in this case.  
\end{case}
\begin{case}[$v(G)>n-1$]

Let $M$ be a maximum matching in $G$ and let $V_M$ be all vertices covered by the matching. As $G$ has no matching covering $n$ vertices and $n$ is even, $2|M| \leq n-2$. 
Then we know by Theorem~\ref{thm:TutteBerge} that there exists a subset of $S\subset V$ such that $|V\setminus V_M|= \q{G\setminus S} -|S|$.

Let $\mathcal{Q}=\{Q_1,\ldots,Q_{q(G\setminus S)}\}$ be the set of all odd components in $G\setminus S$ where we assume without loss of generality that $|Q_1|\geq |Q_2|\geq \ldots\geq |Q_{q(G\setminus S)}|$. Let $I$ be be an arbitrary set of vertices from the odd components such that $|I\cap Q_1|=0$ and $|I\cap Q_i|=1$ for all $2\leq i\leq q(G\setminus S)$. Then clearly we have that $I$ is an independent set (proving \eqref{lemstruc:Cond:2}). Note that we have 
\[|I|= q(G\setminus S) -1.\]
Let $Q = V\setminus (I\cup S)$ be the set of all remaining vertices. Then $S\cup Q \cup I$ is clearly a partition of $V$ and thus $|V|=|S|+|Q|+|I|$. 
By Theorem~\ref{thm:TutteBerge} we know
\begin{align*}
	n-2\geq 2|M|&=|V|-(q(G\setminus S)-|S|)\\
	&= |S|+|Q|+|I| - q(G\setminus S) + |S|\\
	&=2|S|+|Q|+1.
\end{align*} 
We conclude $n-1\geq 2|S|+|Q|$. In case $n-1>2|S|+|Q|$ we move vertices from $I$ to $Q$ until the above holds with equality. As we are in the case $v(G) > n-1$ this proves \eqref{lemstruc:Cond:1}.

To see that \eqref{lemstruc:Cond:3} holds, recall that $I$ contains at most one vertex from each odd component. Every vertex in $Q$ can thus be adjacent to at most one of the vertices in $I$, the one which was in the same component in $G\setminus S$. As we by construction did not remove any vertex from the largest odd component, no vertex in $I$ can have more that $\frac{|Q|}{2}$ neighbors in $Q$. Together with the vertices from $S$ we conclude for $v\in I$ that \[\deg(v)\leq |S|+\frac{|Q|}{2}=\frac{n-1}{2},\]
proving \eqref{lemstruc:Cond:4}. 
\end{case}
\end{proof}
\section{Proof of Lemma~\ref{lem:fconnloss} and Corollary~\ref{cor:lossFv}}
\label{sec:LemlossCorrloss}

\setcounter{case}{0}
Next we proof Lemma~\ref{lem:fconnloss}, recall that 
$$ f(G)=\frac{n-1}{2}v(G) -e(G)$$
and for every connected component $C$ and every vertex $v\in V(C)=S\cup Q\cup I$ we have
\begin{align*}
f(v)=\begin{cases}
\frac{n-1}{4},&\text{if } v\in S,\\
\frac{n-1}{2}-\frac{\deg(v)}{2},&\text{if }v\in Q,\\
0,& \text{if }v\in I.
\end{cases}
\end{align*} 
\\
\textbf{Lemma~\ref{lem:fconnloss}.} \textit{
Let $G$ be a graph without a connected matching of size $n/2$. 
Then $$\sum_{v\in V(G)}f(v)\leq f(G).$$}

\begin{proof}[Proof of Lemma~\ref{lem:fconnloss}]
Let $G$ be a graph, not necessarily connected, and let $C_1,\ldots,C_m$ be the connected components of $G$. Every vertex and every edge is in exactly one component thus we get
\begin{align*}
	f(G)&=\frac{n-1}{2}v(G)-e(G)\\
	&=\sum_{i=1}^{m} \frac{n-1}{2}v(C_i)-e(C_i)\\
	&=\sum_{i=1}^{m}f(C_i).
\end{align*}
If we know that for every component we have $\sum_{v\in v(C_i)}f(v)\leq f(C_i)$, then 
\begin{align*}
	\sum_{v\in V(G)} f(v) = \sum_{i=1}^m\sum_{v\in v(C_i)}f(v)\leq \sum_{i=1}^{m}f(C_i) = f(G).
\end{align*}
We can thus without loss of generality assume $G$ is connected. This allows us to partition $G$ with Lemma~\ref{lem:structure}.
We distinguish two cases.
\begin{case}[$v(G)\leq n-1$]
	In this case all vertices are in $Q$ as by Lemma~\ref{lem:structure} \eqref{lemstruc:Cond:2} we know $I=\emptyset$ and thus $|Q|+2|S|=v(G)$ implies $S=\emptyset$. We clearly have  
	\begin{align*}
	f(G)&=\frac{n-1}{2}v(G) -e(G)= \sum_{v\in V(G)}\left(\frac{n-1}{2}-\frac{\deg(v)}{2}\right)=\sum_{v\in V(G)}f(v).
	\end{align*}
\end{case}

\begin{case}[$v(G)>n-1$]
As in the previous case we rewrite the function $f$
$$f(G) = \frac{n-1}{2}v(G) -e(G)= \sum_{v\in V(G)}\left(\frac{n-1}{2}- \frac{\deg(v)}{2}\right).$$
Let $V(G)= S\cup Q \cup I$. We now distinguish between vertices which are in $Q$, $S$ or $I$.
By definition
\begin{equation}\label{eq:qvertices}
\sum_{v\in Q} \left(\frac{n-1}{2} -   \frac{\deg(v)}{2}\right)    = \sum_{v\in Q} f(v).
\end{equation}
Then we know that $\deg(v)$ for $v\in S$ is at most $v(G)-1$ and $\deg(v)$ for $v\in I$ is at most $|S|$ plus at most $|Q|$ edges in total from the set $I$ to $Q$. Therefore, for vertices in $S$ and $I$ we have the inequalities
\begin{align*}
\sum_{v\in S}\left(  \frac{n-1}{2} -  \frac{\deg(v)}{2}  \right)   &\ge \sum_{v\in S} \left( \frac{n-1}{2} -  \frac{v(G)-1}{2}  \right),\text{ and}\\
\sum_{v\in I}\left(  \frac{n-1}{2} -   \frac{\deg(v)}{2}  \right) &\ge \sum_{v\in I} \left( \frac{n-1}{2} -  \frac{|S|}{2}   \right) - \frac{|Q|}{2}.
\end{align*}
Adding in the fact that $v(G) = |Q| + |I| + |S|$ and $n-1 = |Q|+2|S|$ we get
\begin{align*}
 \frac{n-1}{2} -  \frac{v(G)-1}{2}   &=  \frac{|S|-|I|+1}{2},\text{ and}  \\
 \frac{n-1}{2} -  \frac{|S|}{2}  &=  \frac{|Q|+|S|}{2} . 
\end{align*}
using the fact that $|I| \ge |S|+1$ we obtain
\begin{align*}
\sum_{v\in S\cup I}\left(\frac{n-1}{2}-\frac{\deg(v)}{2}\right)&\geq |S|  \frac{|S|-|I|+1}{2} + |I|  \frac{|Q|+|S|}{2}  - \frac{|Q|}{2} \\&\ge |S| \frac{|Q|+|S|}{2}   \ge |S|\frac{n-1}{4} \label{eq:svertices}.\numberthis
\end{align*}
Putting inequalities \eqref{eq:qvertices} and \eqref{eq:svertices} together finally gives
\begin{align*}
	f(G) &=  \sum_{v\in V(G)}\left( \frac{n-1}{2} - \frac{\deg(v)}{2}\right) \\&\ge \sum_{v\in Q} \left(\frac{n-1}{2} -   \frac{\deg(v)}{2}\right) + |S|\frac{n-1}{4} = \sum_{v\in V(G)} f(v).
\end{align*} 
This finishes the proof of Lemma~\ref{lem:fconnloss}.
\end{case}

\end{proof}
We deduce Corollary~\ref{cor:lossFv} from Lemma~\ref{lem:fconnloss}. For this consider $G$ to be a $k$-edge-colored graph. Recall that $G_i$ is the monochromatic induced subgraph in color $i$.
We denote the associated function from Definition~\ref{def:fG} and~\ref{def:fvW} for the graph $G_i$ in color $i$ with $f_i$.
\begin{proof}[Proof of Corollary~\ref{cor:lossFv}]
	
By definition of $F(G)$ and $f_i(G)$ we have that $F(G)=\sum_{i=1}^{k} f(G_i)$. We conclude that if $F(v)\leq \sum_{i=1}^{k} f_i(v)$, then this together with Lemma~\ref{lem:fconnloss} implies \begin{align*}
\sum_{v\in V(G)}F(v)\leq \sum_{v\in V(G)}\sum_{i=1}^{k} f_i(v)=\sum_{i=1}^{k} \sum_{v\in V(G)}f_i(v)\leq\sum_{i=1}^{k}f_i(G_i)=F(G).
\end{align*} 
In the following we show that $F(v)\leq \sum_{i=1}^{k} f_i(v)$ indeed holds. For this we look at $F(v)$ depending on the class of the vertex $v$. 

For every strong vertex $v$, we know that for at least one color $i$, $v\in S_i$, so we have $f_i(v)=\frac{n-1}{4}$. Hence then $F(v)=f_i(v)\leq \sum_{i=1}^{k} f_i(v)$ by the non-negativity of $f(v)$.\\

For every $Q$-saturated vertex by definition we have $f_i(v)=\frac{n-1}{2}-\frac{\deg_{G_i}(v)}{2}$. This means that for every $Q$-saturated vertex of $G$ 
\begin{align*}
	\sum_{i=1}^{k}f_i(v)&=\sum_{i=1}^{k}\left(\frac{n-1}{2} -\frac{\deg_{G_i}(v)}{2}\right)= k\cdot\frac{n-1}{2} -\frac{\deg_{G}(v)}{2} = F(v),
\end{align*}
since the graphs $G_i$ are edge-disjoint and their union is $G$.\\

For every small vertex we have, by non-negativity of $f(v)$, that $F(v) = 0 \leq \sum_{i=1}^{k}f_i(v)$.

\end{proof}

\section{Conclusion}
\label{sec:conclusion}

In this paper, we provided some insight into the behavior of graphs avoiding connected matchings. We introduced strong properties for large connected components without a matching of size $n/2$. Also these directly imply a better bound for the multicolor Ramsey numbers of paths and cycles.

The analysis of these large connected components is important because of the different extremal constructions that exist for the lower bound of $R_k(P_n)$. The bound by Yongqi et al.~\cite{yongqi2006new} uses large connected components. However, a construction using finite affine planes (see~\cite{Bierbrauer1987123}) shows that similar bounds can be achieved when all colors have small connected components. 

To prove the bound for $R_k(C_n)$, we first bound the Ramsey numbers for connected matchings (Theorem~\ref{thm:main}) and then conclude by applying Lemma~\ref{lem:CycMat}, which itself applies the regularity lemma. Because the regularity lemma is such a strong tool, it should be possible to deduce Ramsey bounds for other structures than the path or even cycle, such as bounded degree trees, by adjusting Lemma~\ref{lem:CycMat} (see~\cite{mota2015ramsey} for similar ideas with three colors). 

Although we now have some consideration for large connected components and the overlap of small components, we do not look at the overlap of large components. Considering this might lead to better bounds. We would be interested to see these ideas used to prove an upper bound matching the lower bound.

\subsection*{Acknowledgements}

We would like to thank Rajko Nenadov for bringing this problem to our attention. We also thank Miloš Trujić for many helpful comments.

\bibliographystyle{abbrv}
\bibliography{improvingramseypath}
\end{document}